\documentclass[11pt]{amsart}
\usepackage{graphicx,amssymb,amsmath,amsthm}
\usepackage{comment, enumerate}
\usepackage{dsfont}
 \numberwithin{equation}{section}
\usepackage{mathrsfs}
\usepackage{epsfig}
\usepackage{lscape}
\usepackage{subfigure}
\usepackage{epstopdf}
\usepackage{color}
\usepackage{caption}
\usepackage{bm}
\usepackage{algorithm}
\usepackage{algpseudocode}

\newcommand{\xkh}[1]{\left(#1\right)}
\newcommand{\dkh}[1]{\left\{#1\right\}}

\newcommand{\nj}[1]{\langle {#1} \rangle}
\newcommand{\innerp}[1]{\langle {#1} \rangle}
\newcommand{\norm}[1]{\|{#1}\|}

\newcommand{\norms}[1]{\|{#1}\|}
\newcommand{\abs}[1]{\lvert#1\rvert}

\newcommand{\argmin}[1]{\mathop{\rm argmin}\limits_{#1}}

\newcommand{\s}{{\mathrm s}}

\newcommand{\R}{{\mathbb R}}

\newcommand{\HH}{{\mathbb H}}

\newcommand{\KK}{{\mathcal K}}
\newcommand{\T}{\top}
\newcommand{\C}{{\mathbb C}}

\newcommand{\x}{{\widehat{\bm{x}}}}

\newcommand{\dist}{{\rm dist}}

\newcommand{\vx}{{\bm x}}
\newcommand{\vy}{{\bm y}}
\newcommand{\vu}{{\bm u}}

\newcommand{\vz}{{\bm z}}
\newcommand{\vb}{{\bm b}}

\newcommand{\ve}{{\bm e}}
\newcommand{\bepsilon}{{\mathbf \epsilon}}

\newcommand{\va}{{\bm a}}

\newcommand{\tx}{{\widetilde{\bm{x}}}}
\newcommand{\tA}{{\widetilde{A}}}

\newcommand{\rank}{{\rm rank}}
\newcommand{\conv}{{\rm conv}}

\newcommand{\spann}{{\rm span}}

\renewcommand{\omega}{\eta}

\newcommand{\RNum}[1]{\uppercase\expandafter{\romannumeral #1\relax}}

\newtheorem{definition}{Definition}[section]
\newtheorem{cor}[definition]{Corollary}

\newtheorem{theorem}[definition]{Theorem}
\newtheorem{conjecture}[definition]{Conjecture}
\newtheorem{lemma}[definition]{Lemma}

\newtheorem{example}[definition]{Example}

\date{}

\begin{document}
\baselineskip 18pt
\bibliographystyle{plain}
\title[uniqueness and stability of solutions]{ Uniqueness and stability for the solution of a nonlinear least squares problem}
\author{Meng Huang}
\thanks{*****}
\address{Department of Mathematics, The Hong Kong University of Science and Technology, Clear
Water Bay, Kowloon, Hong Kong, China}\email{menghuang@ust.hk}

\author{Zhiqiang Xu}
\thanks{Zhiqiang Xu was supported  by  NSFC grant (12025108), by
 by  Beijing Natural Science Foundation (Z180002) and by NSFC grant (12021001).}
\address{LSEC, Inst.~Comp.~Math., Academy of
Mathematics and System Science,  Chinese Academy of Sciences, Beijing, 100091, China
\newline
School of Mathematical Sciences, University of Chinese Academy of Sciences, Beijing
100049, China} \email{xuzq@lsec.cc.ac.cn}
\maketitle

\begin{abstract}
In this paper, we focus on the nonlinear least squares: $\min_{\vx\in \HH^d
}\|\abs{A\vx}-\vb\|$ where $A\in \HH^{m\times d}$, $\vb \in \R^m$ with $\HH\in
\dkh{\R,\C}$ and consider the uniqueness and stability of solutions. Such problem
arises, for instance,   in phase retrieval and absolute value rectification neural
networks.
For the case where $\vb=\abs{A\vx_0}$ for some $\vx_0\in \HH^d$,  many results have been  developed to characterize the uniqueness and stability of solutions.
However,  for
the case where $\vb\neq \abs{A\vx_0}$ for any $\vx_0\in \HH^d$, there is no existing
result for it to the best of our knowledge.
 In this paper, we first focus on the uniqueness of solutions and show for any matrix $A\in \HH^{m \times d}$ there always exists a vector $\vb\in \R^m$  such that the solution is not unique.
 But, in real case, such ``bad'' vectors $\vb$ are negligible, namely, if $\vb\in \R_{+}^m$ does not lie in some measure zero set, then the solution is unique.
We also present some conditions under which  the solution is  unique. For the
stability of solutions, we prove that the solution is never uniformly stable.
 But if we restrict the vectors $\vb$ to any convex set then it is stable.
 \end{abstract}

\section{introduction}

\subsection{Problem setup}
Assume that $A:=[\va_1,\ldots,\va_m]^* \in \HH^{m\times d}$ and
$\vb:=[b_1,\ldots,b_m]^*\in \R^m$ where $\HH\in \{\R,\C\}$. We are interested in the
following program
\begin{equation}\label{eq:least}
\Phi_A(\vb)\,\,:=\,\,\argmin{\vx\in \HH^d} \quad \|\abs{A\vx}-\vb\|^2,
\end{equation}
where $\abs{\cdot}$ is understood to act entrywise. Such model has a rich history in
statistics and is   widely used in phase retrieval (see
\cite{wang2017solving,zhang2016reshaped,
fienup1982phase,netrapalli2015phase,gerchberg1972practical}) and  deep learning
\cite{jarrett2009best, goodfellow2016deep}. Although one has developed many
algorithms to solve (\ref{eq:least}), especially in the randomized setting (meaning
that the matrix $A$ is drawn at random),
 there are very few results about the properties of the program, such as the uniqueness and stability of the solution.

For convenience, we set
\[
\KK_A:=\{\vy=\abs{A\vx}\in \R^m : \vx\in \HH^d\}
\]
and call $\KK_A$ as a {\em phaseless surface} corresponding to $A$.
 When $\vb\in \KK_A$, i.e., $\vb=\abs{A\vx_0}$ for some $\vx_0\in \HH^d$,
 the  recovery of $\vx_0$ from the phaseless observation vector $\vb$ is known as
{\em phase retrieval}. For this case where  $\vb\in \KK_A$, there are
many results for the uniqueness and stability of the solution to (\ref{eq:least}).
For instance, if $m\geq 2d-1$ (resp. $m\geq 4d-4$) then the generic matrix $A\in
\R^{m\times d}$ (resp. $A\in \C^{m\times d}$) suffices to guarantee the uniqueness of  the solution to (\ref{eq:least}) (see e.g. \cite{phase1,phase2,phase3});  moreover, for $\vb\in \KK_A$, the
solution to (\ref{eq:least}) is  always stable for any fixed $d$ (i.e.,  finite-dimensional Hilbert space)
\cite{bandeira2014saving, phase} while it is unstable in any infinite-dimensional
Hilbert space \cite{Rima2017,Jameson}. However, in the {\em noisy phase retrieval },
we often encounter the case where $\vb\notin \KK_A$. In this setting, to our
knowledge,  there is no result concerning the uniqueness and stability of
solutions. Naturally, one may be interested in whether the solution to
(\ref{eq:least}) is unique or stable for any $\vb\in \R^m$, which is the topic  of
this paper.

As said before,  the aim of this paper is to address the uniqueness  and stability of solutions of
the nonlinear least squares problem (\ref{eq:least}). Particularly, we are interested in  the following questions:
\begin{enumerate}[{\rm Question~} I]
\item {\em (Uniqueness of solutions)} Does  there exist a matrix $A\in
    \HH^{m\times d}$ so that the solution to (\ref{eq:least}) is unique up to a
    unimodular constant for {\em all}  the vectors $\vb\in \R^m$?
\item {\em (Conditions for uniqueness)} For which vector $\vb\in \R^m$,  the solution
    to (\ref{eq:least}) is unique?
\item {\em (Stability of solutions)} Is there a matrix $A\in \HH^{m\times
    d}$ and a constant $c$ only depending on $A$ so that
    \[
    \min_{\vx\in \Phi_A{(\vb_1)}, \vy\in \Phi_A{(\vb_2)}}\|\vx-\vy\|\leq c \|\vb_1-\vb_2\|
    \]
    holds for all $\vb_1,\vb_2\in \R^m$?
\end{enumerate}

Note that  if $\vx \in \HH^d$ is a solution to (\ref{eq:least}) then $c \vx$ is also
a solution to (\ref{eq:least}) for any unimodular constant $c$.  Thus,  we say $\vx
\sim \vy$ if $\vx=c\vy $ for some unimodular constant $c$. Let
$\underline{\HH^d}:=\HH^d/\sim$. We shall use $\underline{\vx}$ to denote the
equivalent class containing $\vx$. We say that the solution to (\ref{eq:least}) is
unique if $\Phi_A(\vb)$ only contains one element in $\underline{\HH^d}$. The
distance between $\underline{\vx}$ and $\underline{\vy}$ is defined as
$\|\underline{\vx}-\underline{\vy}\|:=\min_{c\in \HH,\abs{c}=1}\|\vx-c\vy\|$.

 \subsection{Related work}
 \subsubsection{Phase retrieval}

 The most related example to  (\ref{eq:least}) is {\em phase retrieval},
  which aims to recover the signals from the magnitudes of measurements.
The phase retrieval problem arises in many areas, such as   X-ray crystallography
  \cite{harrison1993phase, millane1990phase}, optics \cite{walther1963question},
  astronomical imaging \cite{fienup1987phase}, diffraction imaging \cite{bunk2007diffractive},
  and microscopy \cite{miao2008extending}. In these areas, the phase information of an
  object is lost due to physical limitations of scientific instruments.
  More specifically, suppose that a signal $\vx_0 \in \HH^d$ is measured
   via measurement vectors $\va_i\in \HH^d$ to obtain $b_i=\abs{\nj{\va_i,\vx_0}}, \; i=1,\ldots,m$.
 The phase retrieval problem aims to recover the
 signal $\vx_0$ based on measurement matrix $A:=[\va_1,\ldots,\va_m]^* \in \HH^{m\times d}$ and
 vector  $ \vb:=[b_1,\ldots,b_m]^*\in \R^m$. A natural approach to reconstruct
  $\vx_0$ is to employ  (\ref{eq:least}). Many efficient algorithms have been proposed
  for solving (\ref{eq:least}) with the proviso that $A$ is a Gaussian random matrix, such as Truncated Amplitude Flow \cite{wang2017solving},
 Reshaped Wirtinger Flow \cite{zhang2016reshaped},   Perturbed Amplitude Flow \cite{Gao20}
  and Smoothed Amplitude Flow \cite{Huang2021}.

   We say  a   matrix $A\in \HH^{m\times d}$ has {\em phase retrieval property }
    if one can recover any
   $\underline{\vx_0}\in \underline{\HH}^d$  from $\vb=\abs{A\vx_0}\in \KK_A$. For the real case,
 the matrix $A\in \R^{m\times d}$ has phase retrieval property if and only if $A$ satisfies
 the complement property \cite{bandeira2014saving},  which implies $m\geq 2d-1$ generic
      vectors of $\R^d$ are sufficient to have phase retrieval property.
 For the complex case, Balan, Casazza and Edidin in \cite{phase1} show that
  $A\in \C^{m\times d}$ has phase retrieval property if $m\ge 4d-2$ and $\va_1,\ldots,\va_m$
  are generic vectors in $\C^d$. Lately, Bandeira, Cahill, Mixon and Nelson improve this
  result to $m\ge 4d-4$ generic vectors \cite{phase2}.

Recently, the phase retrieval problem under a generative  prior is studied in
\cite{hand2018phase, shamshad2020compressed} and is termed as {\em deep phase
retrieval}. In such setting, the signal of interest is the output of a generative
model which is a $n$-layer, fully-connected, feed forward neural network with
Rectifying Linear Unit (ReLU) activation functions and no bias terms. To recover the
signal, they consider the empirical risk minimization problem:
\begin{equation} \label{mo:deepPR}
\min_{\vx \in \R^d} \quad \norm{\abs{AG(\vx)}- \abs{AG(\vx_0)}}^2,
\end{equation}
where $G(\vx):=\mbox{ReLU}(W_n\cdots\mbox{ReLU}(W_2 (\mbox{ReLU}(W_1 \vx))))$ with
the weights $W_i \in \R^{k_i \times k_{i-1}}$  and $\mbox{ReLU}(z)=\max(0,z)$. The
results of \cite{hand2018phase, shamshad2020compressed}  show that the objective
function of (\ref{mo:deepPR}) exhibits favorable geometry landscape and does not have
any spurious local minima away from neighborhoods of the true solutions  provided
$A\in \R^{m\times k_n}$ is Gaussian random matrix and $m=\Omega(dn\log(k_1\cdots
k_n))$. A simple observation is that (\ref{mo:deepPR}) has a unique solution up to a
unimodular constant if $A$ has phase retrieval property. That is another  example of
program (\ref{eq:least}) which combine  phase retrieval and deep learning.

\subsubsection{Shallow neural networks}
Another example related to the program (\ref{eq:least}) is shallow neural networks with
{\em absolute value rectification}.  More specifically, given training data
$\dkh{(\va_i,b_i) }_{i=1}^m \in \R^{d}\times \R$, we consider a neural network with
zero hidden unit and a single output with absolute value activation to fit the data.
A natural approach is to minimize the least squares misfit aggregated over the data,
which is in the form of (\ref{eq:least}) exactly. {\em Absolute value rectification}
$g(z)=\abs{z}$  is a generalization of ReLU units. Since the slope is non-zero when
$z$ is negative, it can be used to avoid the dead ReLU problem. Fitting the data with
{\em absolute value rectification}  has several advantages over others activation
functions when taking into account the (sign) symmetry of features
\cite{xu2016structural}. For example, for the object recognition from images, it
makes sense to use absolute value rectification to seek features that are invariant
under a polarity reversal of the input illumination \cite{goodfellow2016deep}.  We
would like to point out that there is an interesting growing literature
\cite{hazan2015beyond, kakade2011efficient, kalai2009isotron, mei2016landscape,
soltanolkotabi2017learning} on learning shallow neural networks with zero hidden unit
and a single output, most of which focus on geometric landscape analysis and the
convergence of gradient-based methods to  the global optimum under various
assumptions. To our knowledge, there is little works considering the uniqueness and
stability of solutions. Since neural networks have achieved remarkable empirical
success  \cite{collobert2008unified, goel2017reliably, oymak2020towards,
du2018gradient} while still lack of theoretical guarantees,  we believe that the
results in our paper are useful in reducing the gap.


\subsection{Our Contribution}

The aim of this paper is trying to answer  Question I, Question II and Question III.
For Question I, we  prove that for {\em arbitrary } matrix $A \in \HH^{m\times d}$
there always exists $\vb\in \R^m$ such that the solution to (\ref{eq:least}) is not
unique, which gives a negative answer for it. We then turn to Question II in the real case.
First, we show that the set of nonuniqueness vectors  $\vb$ is negligible in the nonnegative orthant, i.e.,
for all vectors $\vb\in \R_{+}^m$ except a measure zero set the solution to
(\ref{eq:least}) is unique. Recall that we use $\Phi_A(\vb)$ to denote the solutions
set to (\ref{eq:least}). We next  prove that $\# \Phi_A(\vb)$ is finite provided $A$
satisfies the phase retrieval property.  Finally, we present a sufficient condition, the vector $\vb$ is very
close to the set $\KK_A$,  under which the solution to (\ref{eq:least}) is unique.
 These explain the reason why the solution to (\ref{eq:least}) is often unique in
 many numerical experiments. Although the results only hold in the real case,
it sheds light on the relationship of the vector $\vb$ to uniqueness of the solution.

 Finally, we consider Question
III, i.e., the stability of solutions. We prove that for any $\epsilon>0$ there
always exist $\vb_1, \vb_2\in \R^m$ so that $\mbox{dist}(\Phi_A(\vb_1),\Phi_A(\vb_2))
\ge \norm{\vb_1-\vb_2}/\epsilon$, which means the solution to (\ref{eq:least}) is
never uniformly stable. But if we restrict the vector $\vb$ to some convex sets, then
the solution to (\ref{eq:least}) is stable.


\subsection{Organization}
The paper is organized as follows.
In Section 2, we introduce some notations and lemmas which are useful in this
paper. In Section 3, we present a negative result for Question I and show the
solution to (\ref{eq:least}) is not unique for some vectors $\vb$.  Section 4 is devoted to
establishing several uniqueness results under some appropriate conditions, which gives a
positive answer to Question II. Finally,  Section 5 is concerned with the stability of solutions to  (\ref{eq:least}),
which gives the answers to Question III.

\section{Preliminaries}
In this section, we  introduce a few notations  and lemmas  that will be used in our paper.

\subsection{The best approximation and Chebyshev sets}
Assume that $K\subset \R^m$ is nonempty. For any fixed $\vb\in \R^m$, if $\vy^\#\in \R^m$ satisfies
\[
\|\vy^\#-\vb\|={\min}_{\vy\in K}\|\vy-\vb\|
\]
 then $\vy^\#$ is called a {\em best approximation} to $\vb$ from $K$ and $d(K, \vb):=\|\vy^\#-\vb\|$ is called the distance from $\vb$ to $K$.
  We use $P_K(\vb)$ to denote the set of all best approximations to $\vb$ from $K$. In the context of the best approximation
theory, $K$ is called a {\em Chebyshev set} if each $\vb \in \R^m$ has a unique best
approximation in $K$ (see \cite{best}). The following lemma presents
 a characterization of Chebyshev set in finite-dimensional Hilbert space.
\begin{lemma} (\cite[Theorem 12.7]{best} )\label{th:best}
Assume that $K$ is a nonempty subset of $\R^m$. Then $K$ is a Chebyshev set if and
only if $K$ is closed and convex.
\end{lemma}
The next lemma states that the distance function is nonexpansive for any nonempty set.
\begin{lemma}\label{le:dist}\cite[Theorem 5.3]{best}
Assume that $K\subset \R^m$ is a nonempty set. Then for every pair $\vb,\vb' \in
\R^m$,
\[
\abs{d(K,\vb)-d(K,\vb')}\leq \|\vb-\vb' \|.
\]
\end{lemma}

The following lemma shows that the projection operator onto a Chebyshev set  is also nonexpansive.
\begin{lemma}\label{le:leq1}\cite[Theorem 12.3]{best}
Assume that $K\subset \R^m$ is closed and convex. Then
\[
\|P_K(\vb)-P_K(\vb')\|\,\,\leq \,\, \|\vb-\vb'\|\quad \text{ for all }\quad \vb,\vb'\in \R^m.
\]
\end{lemma}

%
%
%
%
%
%

\subsection{Some results about phase retrieval}

As stated before,  we say  a matrix $A\in \HH^{m\times d}$ has  phase retrieval
property if any $\underline{\vx_0}\in \underline{\HH^d}$ can be recovered from
$\abs{A\vx_0}\in \R^m$.

The following lemma presents a relationship between the  solution
to (\ref{eq:least}) and the best approximation to $\vb$ from  $
\KK_A:=\{\vy=\abs{A\vx}\in \R^m : \vx\in \HH^d\} $.
\begin{lemma}\label{le:one}
Assume that $A\in \HH^{m\times d}$ has phase retrieval property.
For any vector $\vb\in \R^m$, the program (\ref{eq:least}) has a unique solution if and only if the best
approximation to $\vb$ from the  $\KK_A$ has exactly one element, i.e.,
$\#P_{\KK_A}(\vb)=1$.
\end{lemma}
\begin{proof}
We assume that   the best approximation to $\vb$ from $\KK_A$ has exactly one
element. Then there exists a vector $\vb_1 \in \KK_A$ such that
\begin{equation} \label{eq:b1b2}
\norms{\vb-\vb_1} < \norms{\vb-\vb_2} \quad \mbox{for any} \quad \vb_2 \in \KK_A\setminus \{\vb_1\}.
\end{equation}
Since $A$ has phase retrieval property,  there exists a unique  $\underline{\vx_1}\in
\underline{\HH^d}$ such that $\vb_1=\abs{A\vx_1}$.  According to (\ref{eq:b1b2}), we
have
\[ \norms{\abs{A\vx_1}-\vb} < \norms{\abs{A\vx_2}-\vb} \quad \mbox{for any} \quad
\underline{\vx_2} \in \underline{\HH^d}\setminus \{\underline{\vx_1}\},
\]
which implies the solution to (\ref{eq:least}) is unique.

We next assume that  (\ref{eq:least}) has a unique solution. We will show that the
best approximation to $\vb$ from the phaseless surface $\KK_A$ contains only one
element. For the aim of contradiction, we  assume there exist two best approximations
to $\vb$ for $\KK_A$, say $\vb_1$ and $\vb_2$. Then there exist two vectors $\vx_1,
\vx_2 \in \HH^d $ with $\underline{\vx_1}\neq \underline{\vx_2}$ such that
$\vb_1=\abs{A\vx_1},\; \vb_2=\abs{A\vx_2}$ and
\[
\norms{\abs{A\vx_1}-\vb}=\norms{\abs{A\vx_2}-\vb} < \norms{\abs{A\vx}-\vb} \quad \mbox{for any}
\quad \underline{\vx} \in \underline{\HH^d}\setminus \{\underline{\vx_1},\underline{\vx_2}\},
\]
which implies (\ref{eq:least}) has two solutions $\vx_1$ and $\vx_2$.  This contradicts to the assumption.

\end{proof}

For the  case where $\HH=\R$, the matrix $A\in \R^{m\times d}$ has phase retrieval
property if and only if $A$ satisfies the complement property:

\begin{lemma}\cite{bandeira2014saving}\label{le:comp}
The matrix $A:=[\va_1,\ldots,\va_m]^\T\in \R^{m\times d}$ has phase retrieval property in $\R^d$ if and only if for every
$I\subset \{1,\ldots,m\}$, either ${\rm span}\{\va_j:j\in I\}=\R^d$ or ${\rm
span}\{\va_j:j\in I^c\}=\R^d$.
\end{lemma}

The following lemma shows that, for the real case, any solution to (\ref{eq:least})
satisfies a fixed-point equation.

\begin{lemma}\cite{huangxu} \label{fixed-point real}
Suppose that $A:=[\va_1,\ldots,\va_m]^\T\in \R^{m\times d}$ and  $\vb\in\R_+^m$. Assume that $\x$ is a
solution to (\ref{eq:least}). Then $\x$ satisfies the following fixed-point equation:
\begin{equation*}
  \x=(A^\T A)^{-1}A^\T (\vb\odot\s(A\x)),
\end{equation*}
where $\odot$ denotes the Hadamard product and $\s(A\x):=
\left(\frac{\nj{\va_1,\x}}{\abs{\nj{\va_1,\x}}},\ldots,\frac{\nj{\va_m,\x}}{\abs{\nj{\va_m,\x}}}\right)$
for any $\x\in\R^d$. Here, $\frac{\nj{\va_j,\x}}{\abs{\nj{\va_j,\x}}}=1$ is adopted
if $\nj{\va_j,\x}=0$.
\end{lemma}

%

\section{The non-uniqueness of solutions to (\ref{eq:least})}
The aim of this section is to answer Question I by showing that the solution to
(\ref{eq:least}) is nonunique  for some vectors $\vb\in \R^m$. We state  the main
result of this section as follows.

\begin{theorem}\label{th:main}
Assume that $m, d$ are positive integers. For arbitrary matrix $A\in \HH^{m\times
d}$, there exists $\vb\in \R^m$ so that the solution to (\ref{eq:least}) is not unique
where $\HH\in \{\R,\C\}$.
\end{theorem}

To prove this theorem, according to Lemma \ref{le:one}, it is enough  to show the set
$\KK_A$ is not a Chebyshev set. From Lemma \ref{th:best}, we can do it by showing the
set $\KK_A$ is not a convex set.

\begin{lemma}\label{th:nonc}
Assume that $A\in \HH^{m\times d}$ has phase retrieval property.  Then the set
$\KK_A\subset \R^m$ is non-convex.
\end{lemma}
\begin{proof}
We first prove it in the real case where $\HH=\R$.
For the aim of contradiction, we assume that $\KK_A$ is convex.
 Let $A:=[\va_1,\ldots,\va_m]^\T \in \R^{m\times
d}$ with
\[
\va_j=(a_{j,1},\ldots,a_{j,d})^\T\in \R^d.
\]
 Without loss of
generality, we assume that $\va_j=\ve_j,j=1,\ldots,d$. Since $A$ has phase retrieval
property,  there exists $j_0\in \dkh{d+1,\ldots, m}$ so that $\|\va_{j_0}\|_0\geq 2$.
Without loss of generality, we assume that $j_0=d+1$ and the $d$-th component of
$\va_{d+1}$ is positive, i.e., $a_{d+1,d}>0$. For any $\vx=(x_1,\ldots,x_d)\in
{\mathbb R}^d$, we have $\abs{\innerp{\va_j,\vx}}=\abs{x_j}, \; j=1,\ldots,d$. Then
there exist $\epsilon>0$ and $r_0>0$ so that $\innerp{\va_{d+1},\vx}>0$ provided
$\vx\in [-\epsilon,\epsilon]^{d-1}\times [r_0,\infty]$.
 Let $\vx':=(-x_1,\ldots,-x_{d-1},x_d)$.
Since $\KK_A$ is convex, we have
\begin{equation}\label{eq:2CA}
\frac{1}{2}\abs{A\vx}+\frac{1}{2}\abs{A\vx'}\in \KK_A,
\end{equation}
where $\vx,\vx'\in [-\epsilon,\epsilon]^{d-1}\times [r_0,\infty]$.
Note that the first $d+1$ entries of $\frac{1}{2}\abs{A\vx}+\frac{1}{2}\abs{A\vx'}$
 is $(\abs{x_1},\ldots,\abs{x_d},a_{d+1,d}x_d)$
 provided  $\vx,\vx'\in [-\epsilon,\epsilon]^{d-1}\times [r_0,\infty]$.
 Here, we use $\innerp{\va_{d+1},\vx}>0$
 and $\innerp{\va_{d+1},\vx'}>0$ if  $\vx,\vx'\in [-\epsilon,\epsilon]^{d-1}\times
[r_0,\infty]$. According to (\ref{eq:2CA}), for any $\vx\in
[-\epsilon,\epsilon]^{d-1}\times [r_0,\infty]$, there exists
$\hat{\vx}=(\hat{x}_1,\ldots,\hat{x}_d)$ with $\hat{x}_d>0$ so that
\begin{equation}\label{eq:jue}
\abs{\hat{x}_j}=\abs{{x}_j}, j=1,\ldots,d, \quad \abs{\innerp{\va_{d+1},\hat{\bm{x}}}}={a_{d+1,d}x_d}.
\end{equation}
Note that $x_d>r_0>0$. Combining $\abs{x_d}=\abs{\hat{x}_d}$ and $\hat{x}_d>0$, we
have $\hat{x}_d=x_d$. Since $\abs{\hat{x}_j}=\abs{x_j}\leq \epsilon$ for all
$j=1,\ldots,d-1$, then the choice of $r_0$ implies $\innerp{\va_{d+1},\hat{\vx}}>0$.
 According to (\ref{eq:jue}),  we have
 \[
 {\innerp{\va_{d+1},\hat{\vx}}}={a_{d+1,d}x_d},
 \]
 which implies that
 \begin{equation}\label{eq:ax0}
 a_{d+1,1}\hat{x}_1+\cdots+a_{d+1,d-1}\hat{x}_{d-1}=0
 \end{equation}
holds for any $(x_1,\ldots,x_{d-1})\in [-\epsilon,\epsilon]^{d-1}$. Combining
(\ref{eq:ax0}) and $\abs{\hat{x}_j}=\abs{x_j}$, we obtain that
$a_{d+1,1}=\cdots=a_{d+1,d-1}=0$, which contradicts to $\|\va_{d+1}\|_0\geq 2$.

We next turn to the complex case where $\HH=\C$. For the aim of contradiction, we assume that
$\KK_A$ is convex.
 Without loss of generality, we assume that $\va_j=\ve_j,j=1,\ldots,d$.
  Since $A \in \C^{m\times d}$ has phase retrieval property,  there exist distinct
$j_0, k_0\in \dkh{d+1,\ldots, m}$ so that  $\abs{a_{j_0,d-1}}^2+\abs{a_{j_0,d}}^2\neq
0$ and $\abs{a_{k_0,d-1}}^2+\abs{a_{k_0,d}}^2\neq 0$. Otherwise, one can not recover
the vector in the form of $(0,\ldots,0,x_{d-1},x_d)\in \C^d$. Without loss of
generality, we assume that $j_0=d+1,k_0=d+2$ and $a_{d+1,d}=a_{d+2,d}=1$.
 We assume that $\vx=(0,\ldots,0,r'e^{-i\theta_{1}},x_d)$ and
$\vx'=(0,\ldots,0,r'e^{-i\theta_{2}},x_d)$. Here, $\theta_1,\theta_2\in [0,2\pi)$ and
$r'>0, x_d>0$ are fixed constants.
 We have
\begin{equation}\label{eq:ccon}
\frac{1}{2}\abs{A\vx}+\frac{1}{2}\abs{A\vx'}\in \KK_A.
\end{equation}
A simple calculation shows that the first $d+2$ entries of
$\frac{1}{2}\abs{A\vx}+\frac{1}{2}\abs{A\vx'}$ are $(\underbrace{0,\ldots,0}_{d-2},
r', x_d, s_1,s_2)^\T$ where $s_1:=s_1(\theta_1,\theta_2)=\frac{1}{2}
(\abs{a_{d+1,d-1}r'e^{-i\theta_{1}}+x_d}+\abs{a_{d+1,d-1}r'e^{-i\theta_{2}}+x_d}))$
and $ s_2:=s_2(\theta_1,\theta_2)=\frac{1}{2}
(\abs{a_{d+2,d-1}r'e^{-i\theta_{1}}+x_d}+\abs{a_{d+2,d-1}r'e^{-i\theta_{2}}+x_d}))$.
 According to (\ref{eq:ccon}), there exists $\hat{\vx}=(0,\ldots,0,\hat{x}_{d-1},
\hat{x}_d)$ with $\hat{x}_d>0$ so that
\[
\abs{\hat{x}_{d-1}}=\abs{x_{d-1}}=r',\quad \abs{\hat{x}_d}={x_d},\quad
\abs{a_{d+1,d-1}\hat{x}_{d-1}+\hat{x}_d}=s_1, \quad
\abs{a_{d+2,d-1}\hat{x}_{d-1}+\hat{x}_d}=s_2.
\]
Since $\hat{x}_d>0$ and $\abs{\hat{x}_d}={x_d}$, we have $\hat{x}_d=x_d$. The
$\abs{\hat{x}_{d-1}}=r'$ implies $\hat{x}_{d-1}=r'e^{i\theta'}$ for some $\theta'\in
\R$. So, $(\abs{a_{d+1,d-1}\hat{x}_{d-1}+\hat{x}_d},
\abs{a_{d+2,d-1}\hat{x}_{d-1}+\hat{x}_d})$ is a one dimensional manifold with respect
to $\theta'$ while $(s_1,s_2)$ is two dimensional manifold with respect to $\theta_1$ and
$\theta_2$, which is a contradiction.
\end{proof}

We next present the proof of Theorem \ref{th:main}.

\begin{proof}[Proof of Theorem \ref{th:main}]
We divide the proof into two cases:

{\bf Case 1:} The matrix $A$ does not have phase retrieval property.  From the
definition of phase retrievable, there exist two vectors $\vx_1,\vx_2 \in \HH^d$ with
$\underline{\vx_1}\neq \underline{\vx_2}$  such that $\abs{A\vx_1}=\abs{A\vx_2}$. Let
$\vb:=\abs{A\vx_1}$. Then (\ref{eq:least}) has two solutions  $\vx_1,\vx_2 $ for such
vector $\vb$. The conclusion holds.

{\bf Case 2:} The matrix $A$ has phase retrieval property.
 According to Lemma \ref{th:nonc}, $\KK_A$ is non-convex,
 which means the set $\KK_A$ is not a Chebyshev set by Lemma  \ref{th:best}.
 Thus, the conclusion immediately  follows from Lemma \ref{le:one}.
\end{proof}

The next result shows that, in the real case, $\#\Phi_A(\vb)$ is finite, i.e., the
solutions to (\ref{eq:least}) are  finite.
\begin{theorem}
Assume that $A\in \R^{m\times d}$ has phase retrieval property. Then for any vector
$\vb\in \R^m$, the number of solutions to (\ref{eq:least}) is finite.
\end{theorem}
\begin{proof}
Since $A$ is phase retrievable, it suffices to show the number of the best
approximations to  $\vb$ from $\KK_A$ is finite. Note that
\[
\KK_A=\dkh{\vy:\vy=\abs{A\vx}, \vx\in \R^d}=\cup_{S\subset \dkh{1,\ldots, m}}\dkh{\vy\ge 0:\vy=A_{S}\vx, \vx\in \R^d}.
\]
Here, the $i$th row of $A_S\in \R^{m\times d}$  is defined as follows
\begin{equation*}
		(A_S)_i:=\left\{
		\begin{array}{cl}
			\va_i^\T, & i\in S,\\
			-\va_i^\T, & \mbox{otherwise}.
		\end{array}
		\right.
\end{equation*}
 Since the
set $\dkh{\vy\ge 0:\vy=A_{S}\vx, \vx\in \R^d}$ is convex, it means the best
approximation to  $\vb$ from $\dkh{\vy\ge 0:\vy=A_{S}\vx, \vx\in \R^d}$ is unique.
Note that the number $\abs{S} \le 2^m$. Hence, the number of the best approximations
to  $\vb$ from $\KK_A$ is at most $2^m$, which is finite. This completes the proof.
\end{proof}

\section{The conditions of uniqueness of solutions to (\ref{eq:least})}
In this section,  we focus on  Question II: {\em For which vector $\vb\in \R^m$ the
solution to (\ref{eq:least}) is unique?} We will present several sufficient
conditions for it. Throughout this section, we assume that $\HH=\R$.

\subsection{Almost all the vectors $\vb\in \R_+^m$. }
The following theorem shows that for almost all the vectors
$\vb\in \R_{+}^m$ the solution to (\ref{eq:least}) is unique provided $A\in
\R^{m\times d}$ has phase retrieval property.
\begin{theorem}\label{th:almost}
Suppose that $A\in \R^{m\times d}$ has phase retrieval property in $\R^d$. Then for
all vectors $\vb\in \R_{+}^m$ except  for a measure zero set,  the program
(\ref{eq:least}) has a unique solution.
\end{theorem}
Theorem \ref{th:almost} only considers the real case. We conjecture a similar result
holds for complex case:
\begin{conjecture}
Suppose that $A\in \C^{m\times d}$ has phase retrieval property in $\C^d$. Then for
all vectors $\vb\in \R^m$ except  for a measure zero set,  the program
(\ref{eq:least}) has a unique solution.
\end{conjecture}

 Before presenting the proof of Theorem \ref{th:almost},
 we introduce the following lemma, which will be used in the proof
of Theorem \ref{th:almost}.
\begin{lemma}\label{le:zeroset}
Assume that $A:=[\va_1,\ldots,\va_m]^\T\in \R^{m\times d}$ has phase retrieval
property.
 Set
\[
P_{\epsilon',\epsilon}(\vz):=
\norm{A^\T D_{\epsilon'}\vz}^2-\norm{A^\T D_\epsilon\vz}^2,
\]
where $\vz\in \R^m$, $\epsilon,\epsilon'\in \{1,-1\}^m$, $D_\epsilon={\rm
Diag}(\epsilon_1,\ldots,\epsilon_m)$ and $D_{\epsilon'}={\rm
Diag}(\epsilon'_1,\ldots,\epsilon'_m)$. Then $P_{\epsilon,\epsilon'}(\vz)\not\equiv0$
provided $\epsilon\neq \pm \epsilon' $.
\end{lemma}
\begin{proof}
 A simple observation is that $P_{\epsilon',\epsilon}$ is a  polynomial
with respect to $\vz\in \R^m$. To obtain the conclusion, it is enough to show
$P_{\epsilon',\epsilon}(z_0)\neq 0$ for some $\vz_0\in \R^m$. Without loss of
generality, we assume that $\epsilon'=\ve:=(1,\ldots,1)$. Hence, $\epsilon\neq \pm
\ve$. For convenience, we set $P_\epsilon(\vz):=P_{\ve,\epsilon}(\vz)$.
A simple calculation leads to
\[
P_\epsilon(\vz)=\sum_{j\neq k}z_jz_k \innerp{\va_j,\va_k}-
\sum_{j\neq k}\epsilon_j\epsilon_kz_jz_k \innerp{\va_j,\va_k},
\]
where $\vz=(z_1,\ldots,z_m)\in \R^m$. Set $I:=\{j:\epsilon_j=1\}$ and
$I^c=\{1,\ldots,m\}\setminus I$. Then $I\neq \emptyset$ and $I^c\neq \emptyset$ due to
$(\epsilon_1,\ldots,\epsilon_m)\neq \pm (1,\ldots,1) $.
Since the matrix $A$ has phase retrieval property, according to Lemma
\ref{le:comp}, we have ${\rm span}\{\va_j:j\in I\}=\R^d$ or ${\rm span}\{\va_j:j\in
I^c\}=\R^d$. Without loss of generality, we assume that ${\rm span}\{\va_j:j\in
I^c\}=\R^d$ and that $\epsilon_1=1$.  Since ${\rm span}\{\va_j:j\in I^c\}=\R^d$,
  there exists $j_0\in I^c$ such that
$\innerp{\va_1,\va_{j_0}}\neq 0$. Noting that $j_0\in I^c$, we have
$\epsilon_1\epsilon_{j_0}=-1$. Set
$\vz_0:=(1,0,\ldots,0,\underbrace{-1}_{j_0},0,\ldots,0)\in \R^m$. Then
$\abs{P_\epsilon(\vz_0)}=2\abs{\innerp{\va_1,\va_{j_0}}}\neq 0$ which implies
$P_\epsilon(\vz)\not\equiv 0$. More precisely,  $P_\epsilon(\vz)$ is a nonzero homogeneous polynomial
with respect to $\vz\in \R^m$.
\end{proof}

We are now ready to prove the main result in this subsection.
\begin{proof}[Proof of Theorem \ref{th:almost}]
Since $A$ has phase retrieval property, it then follows from Lemma
\ref{le:comp} that ${\rm rank}(A)=d$. We assume that the
singular decomposition of $A$ is $A=UDV^\T$ where $D \in \R^{d\times d}$ is an
invertible diagonal matrix. Set $\tA:=AVD^{-1}$ and $\tx:=D V^\T \vx$. Note that
\[
\norm{\abs{A\vx}-\vb}=\norm{\abs{\tA\tx}-\vb}.
\]
Hence,  the program (\ref{eq:least}) has a unique solution iff $\argmin{\vx\in
\R^d}\norm{\abs{\tA\vx}-\vb}$ has a unique solution.  Observe that $\tA^\T\tA=I$. So,
to this end, it is enough to  consider the case where  $A^\T A =I$.

Assume that $\vb_0\in \R^m$ so that $\Phi_A(\vb_0)$ contains at least two elements,
i.e., the program (\ref{eq:least}) has at least two solutions for the vector $\vb_0$.
 We claim that $\vb_0 $ satisfies $P_{\epsilon',\epsilon}(\vb_0)=0$ for some $\epsilon', \epsilon \in \{-1,1\}^m$ with $\epsilon\neq
\pm \epsilon'$, where
\begin{equation*}
P_{\epsilon',\epsilon}(\vz):=\norm{A^\T D_{\epsilon'} \vz}^2-\norm{A^\T
D_\epsilon\vz}^2.
\end{equation*}
 According to Lemma \ref{le:zeroset}, $P_{\epsilon', \epsilon}(\vb_0)$ is a nonzero homogeneous polynomial
with respect to $\vb_0 \in \R^m$. It means that the Lebesgue  measure of the set
$\bigcup_{\epsilon\neq \pm\epsilon' }\{\vb_0 \in \R^d:P_{\epsilon',\epsilon}(\vb_0)=0\}$
is 0. This establishes the conclusion.

It remains to prove the claim that  $P_{\epsilon',\epsilon}(\vb_0)=0$.
 Assume that $\vx_1$ and $\vx_2$ are two global
solutions to (\ref{eq:least}) with $\underline{\vx_1}\neq \underline{\vx_2}$, namely, $\vx_1 \neq \pm \vx_2$. Set
\[
S_1:=\dkh{i: \va_i^\T \vx_1 \ge 0} \qquad \mbox{and} \qquad S_2:=\dkh{i: \va_i^\T \vx_2 \ge 0}.
\]
According to Lemma \ref{fixed-point real}, we obtain
\[
\vx_1=A^\T D_{S_1}\vb_0 \qquad \mbox{and} \qquad   \vx_2=A^\T D_{S_2}\vb_0.
\]
Here, $D_S:={\rm Diag}(\delta_S(1),\ldots,\delta_S(m))\in {\mathbb R}^{m\times m}$,
\begin{equation*}
		\delta_S(j):=\left\{
		\begin{array}{cl}
			1, & j\in S,\\
			-1, & \mbox{otherwise}.
		\end{array}
		\right.
\end{equation*}
 Using the notations above, we have
\[
\abs{A\vx_1}=D_{S_1}(AA^\T D_{S_1}\vb_0) \qquad \mbox{and} \qquad
 \abs{A\vx_2}=D_{S_2}(AA^\T D_{S_2}\vb_0).
\]
Since  $\vx_1$ and $\vx_2$ are two global solutions to (\ref{eq:least}),  it gives
\[
\norm{D_{S_1}(AA^\T D_{S_1}\vb_{S_1})-\vb_0}^2=\norm{D_{S_2}(AA^\T D_{S_2}\vb_{S_2})-\vb_0}^2,
\]
which is equivalent to
\begin{equation} \label{eq:equilvalue0}
\norm{AA^\T D_{S_1}\vb_0-D_{S_1}\vb_0}^2=\norm{AA^\T D_{S_2}\vb_0-D_{S_2}\vb_0}^2.
\end{equation}
A simple calculation shows that
\begin{equation}\label{eq:AA}
\begin{aligned}
\norm{AA^\T D_{S_1}\vb_0-D_{S_1}\vb_0}^2&=\nj{AA^\T D_{S_1}\vb_0,AA^\T D_{S_1}\vb_0}
-2\nj{AA^\T D_{S_1}\vb_0,D_{S_1}\vb_0}+\norm{D_{S_1}\vb_0}^2 \\
&= \nj{A^\T AA^\T D_{S_1}\vb_0,A^\T D_{S_1}\vb_0}-2\nj{A^\T D_{S_1}\vb_0,A^\T D_{S_1}\vb_0}
+\norm{D_{S_1}\vb_0}^2 \\
&=  -\norm{A^\T D_{S_1}\vb_0}^2+\norm{\vb_0}^2,
\end{aligned}
\end{equation}
where we use the fact $A^\T A =I$. Combining (\ref{eq:equilvalue0}) and
(\ref{eq:AA}), we obtain
\begin{equation} \label{eq:keyformula}
\norm{A^\T D_{S_1}\vb_0}^2=\norm{A^\T D_{S_2}\vb_0}^2.
\end{equation}
Take $\epsilon'=(\delta_{S_1}(1),\ldots,\delta_{S_1}(m))$ and
$\epsilon=(\delta_{S_2}(1),\ldots,\delta_{S_2}(m))$.  Since $\vx_1\neq \pm \vx_2$, we
know $\epsilon\neq \pm \epsilon' $.
From  (\ref{eq:keyformula}), we have $P_{\epsilon',\epsilon}(\vb_0)=0$, as claimed.
\end{proof}

\subsection{The $\vb\in \R^m$ is  close to the set $\KK_A$.}
In this subsection, we show if the vector $\vb$ is  close to the set $\KK_A=\{\vy=\abs{A\vx} :
\vx\in \R^d\}$ then the solution to (\ref{eq:least}) is unique. For convenience, we
introduce the definition of strong complement property which was firstly introduced
in \cite{bandeira2014saving} (see also \cite{SRIP}).
\begin{definition} \cite[Definition 17]{bandeira2014saving}
We say the matrix $A\in \R^{m\times d}$satisfies the $\sigma$-strong complement property if
\[
\max\dkh{\lambda_{\min} (A_{[\Gamma]}^\T A_{[\Gamma]}), \lambda_{\min} (A_{[\Gamma^c]}^\T A_{[\Gamma^c]})} \ge \sigma^2
\]
for every $\Gamma\subset\dkh{1,\ldots,m}$, where $A_{[\Gamma]}:=[\va_j:j\in
\Gamma]^\T$ denotes the sub-matrix of $A$.
\end{definition}

We next present a sufficient condition under which  the solution to (\ref{eq:least})
is unique.
\begin{theorem} \label{th:clouniq}
Assume that $A\in \R^{m\times d}$ has $\beta\sigma_{\max}(A)$-strong complement
property for some $\beta\in (0,1)$, where $\sigma_{\max}(A)$ is the largest singular
value of $A$. Suppose $\vb=\abs{A\vx_0}+\eta\in {\mathbb R}_+^m$ for some $\vx_0\in
\R^d$ and $\eta \in \R^m$.  If $\norm{\eta}\le \beta^2\lambda$ then the program
(\ref{eq:least}) has a unique solution where  $\lambda:=\min_i\dkh{\abs{A\vx_0}_i:
i=1,\ldots,m}$.
\end{theorem}
\begin{proof}
We first consider the case where  $A^\T A =I$. Note that  $\sigma_{\max}(A)=1$.  Thus the matrix $A$ has $\beta$-strong complement property.
It then follows from  Lemma \ref{le:comp} that $A$ has phase retrievable property.
To obtain the conclusion, according to Lemma \ref{le:one}, it is enough  to show the best approximation to $\vb$ from $\KK_A$ is
unique.

Recognize that $\vb=\abs{A\vx_0}+\eta$.
Let  $S^*:=\dkh{i: \va_i^\T \vx_0 \ge 0}$. Then
\begin{equation}\label{eq:bAS}
\vb=D_{S^*}A\vx_0+\eta.
\end{equation}
Here, $D_{S^*}:={\rm Diag}(\delta_{S^*}(1),\ldots,\delta_{S^*}(m))\in {\mathbb
R}^{m\times m}$ with
\begin{equation*}
		\delta_{S^*}(j):=\left\{
		\begin{array}{cl}
			1, & j\in S^*, \\
			-1, & \mbox{otherwise}.
		\end{array}
		\right.
\end{equation*}
 For convenience, we set  $\KK_S:=\dkh{\vy \in \R^m :\vy=D_SA\vx, \vx\in \R^d}\cap
\R_+^m$.
A simple observation  is that
\[
\KK_A=\dkh{\vy:\vy=\abs{A\vx}, \vx\in \R^d}=\bigcup_{S\subset \dkh{1,\ldots,m}}
\KK_S.
\]
 We claim that,  for any $S\subset \{1,\ldots,m\}$ with $S\neq S^*$, we have
\begin{equation} \label{cla:un}
d(\KK_{S^*},\vb)\,\, <\,\, d(\KK_S,\vb),
\end{equation}
which implies  that the best approximation to $\vb$ from $\KK_A$ is unique. We arrive
at the conclusion.

We next prove the claim \eqref{cla:un}. Let  $P_S(\vb) \in \R^m$ be the projection of
vector $\vb\in\R^m$ onto the subspace $\dkh{\vy\in\R^m:\vy=D_{S}A\vx,  \vx \in
\R^d}$. Note that $A^\T A=I$. A simple calculation leads to
 $$P_S(\vb)=D_SA (D_SA)^\T\vb.$$
For the set $S^*$, we have
\begin{equation*}
\begin{aligned}
P_{S^*}(\vb)&=D_{S^*}A (D_{S^*}A)^\T\vb=D_{S^*}A(D_{S^*}A)^\T(D_{S^*}A\vx_0+\eta)\\
&=D_{S^*}A\vx_0+D_{S^*}A(D_{S^*}A)^\T\eta.
\end{aligned}
\end{equation*}
A simple observation is that the subspace $\dkh{\vy\in\R^m:\vy=D_{S^*}A\vx,  \vx \in
\R^d}$ contains the point $D_{S^*}A\vx_0\in \R_+^m$, which implies the set
$\dkh{\vy\in \R^m:\vy=D_{S^*}A\vx}\cap \R_+^m$ is non-empty. We next prove
$P_{S^*}(\vb)\in \KK_{S^*}$. To this end, we only need to show $P_{S^*}(\vb)\in
\R_+^m$.  Note that
\begin{equation}\label{eq:PSb}
P_{S^*}(\vb) =D_{S^*}A \vx_0+ D_{S^*}A (D_{S^*}A)^\T \eta =\abs{A \vx_0}+ D_{S^*}A (D_{S^*}A)^\T \eta.
\end{equation}
  It is sufficient to prove $\norm{D_{S^*}A (D_{S^*}A)^\T \eta}_{\infty} \le \lambda:=\min_{i} \dkh{\abs{A \vx_0}_i}$.
Note that
\begin{equation*}
\begin{aligned}
\norm{D_{S^*}A (D_{S^*}A)^\T \eta}_{\infty}& \le \max_{1\le i \le m}
\norms{(D_{S^*}A (D_{S^*}A)^\T)_{i}}\norms{\eta}\\
&=
\norms{\eta} \cdot \max_{1\le i \le m}  (\sum_{k=1}^m \abs{\nj{\va_k,\va_i}}^2)^{1/2}
=\max_{1\le i \le m}  \norms{\va_i} \norms{\eta} ,
\end{aligned}
\end{equation*}
where $(D_{S^*}A (D_{S^*}A)^\T)_{i}$ denotes the $i$-th row of $D_{S^*}A
(D_{S^*}A)^\T$ and the last equality follows from $A^\T A= I$. Since the columns of
$A$ are orthonormal, we have $\norms{\va_i}^2  \le 1$ for all $i$. Thus,
\[
\norm{D_{S^*}A (D_{S^*}A)^\T \eta}_{\infty} \le \norm{\eta} \le \lambda,
\]
where we use the condition $\norm{\eta} \le \beta^2 \lambda$ with $\beta\le 1$ in the
last inequality. This immediately gives $P_{S^*}(\vb)\in \R_+^m $.

Combining  (\ref{eq:bAS}), (\ref{eq:PSb}) and $P_{S^*}(\vb)\in \KK_{S^*}$, we have
\begin{equation} \label{eq:hsstar}
d(\KK_{S^*},\vb)=\norm{\vb-P_{S^*}(\vb)}=\norm{(I-D_{S^*}A(D_{S^*}A)^\T)\eta} \le \norm{\eta},
\end{equation}
where the inequality comes from the fact that $I- D_{S^*}A(D_{S^*}A)^\T$ is an orthogonal projection matrix.

Next,  we turn to evaluate $d(\KK_S,\vb)$. For any fixed $S\subset \dkh{1,\ldots,m}$,
define $\Gamma:=(S \backslash S^*) \cup (S^* \backslash S)$.
 If $S\neq S^*$ then $\Gamma\neq \emptyset$.
Since $A$ has $\beta$-strong complement property, we have
\begin{equation}\label{eq:strcompp}
\max\dkh{\lambda_{\min} (A_{[\Gamma]}^\T A_{[\Gamma]}), \lambda_{\min} (A_{[\Gamma^c]}^\T A_{[\Gamma^c]})} \ge \beta^2.
\end{equation}
Without loss of generality, we assume  $\lambda_{\min} (A_{[\Gamma^c]}^\T
A_{[\Gamma^c]}) \ge \beta^2$. Note that the subspaces spanned by $D_SA$ and
$D_{S^c}A$ are the same. Hence,  if $\lambda_{\min} (A_{[\Gamma]}^\T A_{[\Gamma]})
\ge \beta^2$ then we only need to replace the subset $S$ and $\Gamma$ by $S^c$ and
$\Gamma^c:=(S^c \backslash S^*) \cup (S^* \backslash S^c)$, respectively. Recall that
\[
P_S(\vb)=D_SA(D_SA)^\T(D_{S^*}A\vx_0+\eta).
\]
A simple observation is
\[
(D_SA)^\T D_{S^*}A=I-2A_{[\Gamma]}^\T A_{[\Gamma]},
\]
where $(A_{[\Gamma]})_i=\va_i^\T$ if $i\in \Gamma$ and $(A_{[\Gamma]})_i=\bm{0}$ if $i\notin \Gamma$.
Then
\begin{equation*}
\begin{aligned}
P_S(\vb)&=D_SA(D_SA)^\T(D_{S^*}A\vx_0+\eta)=D_SA(I-2A_{[\Gamma]}^\T A_{[\Gamma]})\vx_0
+D_SA(D_SA)^\T\eta\\
&= D_SA\vx_0 - 2 D_SA A_{[\Gamma]}^\T A_{[\Gamma]} \vx_0+D_SA(D_SA)^\T\eta \\
&= D_{S^*}A \vx_0 -2  D_{S^*}A_{{[\Gamma]}}\vx_0- 2 D_SA A_{[\Gamma]}^\T A_{[\Gamma]} \vx_0
+D_SA(D_SA)^\T\eta,
\end{aligned}
\end{equation*}
 which implies
\begin{equation}\label{eq:hs}
\begin{aligned}
d(\KK_S,\vb)&\ge \norm{\vb-P_S(\vb)}=\norm{(I-D_SA(D_SA)^\T)\eta+
2  D_{S^*}A_{{[\Gamma]}}\vx_0+2 D_SA A_{[\Gamma]}^\T A_{[\Gamma]} \vx_0} \\
&\ge  2\norm{D_{S^*}A_{{[\Gamma]}}\vx_0+ D_SA A_{[\Gamma]}^\T A_{[\Gamma]} \vx_0}-
\norm{(I-D_SA(D_SA)^\T)\eta}\\
&= 2\norm{D_{S^*}A_{{[\Gamma]}}\vx_0+ D_SA (D_{S^*}A_{{[\Gamma]}})^\T D_{S^*}A_{{[\Gamma]}} \vx_0}
-\norm{(I-D_SA(D_SA)^\T)\eta}.
\end{aligned}
\end{equation}
Let $\hat{\vb}:=\abs{A\vx_0}$. Then
$D_{S^*}A_{{[\Gamma]}}\vx_0=\hat{\vb}_{[\Gamma]}$. Noting that
$D_S A_{{[\Gamma]}}=-D_{S^*} A_{{[\Gamma]}}$, we have
\begin{equation} \label{eq:medlow}
\begin{aligned}
\norm{D_{S^*}A_{{[\Gamma]}}\vx_0+ D_SA (D_{S^*}A_{{[\Gamma]}})^\T
D_{S^*}A_{{[\Gamma]}} \vx_0}
 & =  \norm{\hat{\vb}_{[\Gamma]}+ D_SA_{S}(D_{S^*}A_{{[\Gamma]}})^\T \hat{\vb}_{[\Gamma]}}\\
& \ge  \norm{\hat{\vb}_{[\Gamma]}-D_{S^*}A_{{[\Gamma]}} (D_{S^*}A_{{[\Gamma]}})^\T \hat{\vb}_{[\Gamma]}}.
\end{aligned}
\end{equation}
From strong complement property (\ref{eq:strcompp}), we know
\begin{equation*}
\norms{D_{S^*}A_{{[\Gamma]}} (D_{S^*}A_{{[\Gamma]}})^\T }=\norms{(D_{S^*}A_{{[\Gamma]}})^\T
D_{S^*}A_{{[\Gamma]}} }=\norms{A_{ [\Gamma]}^\T A_{ [\Gamma] }   }
=\norms{I-A_{[\Gamma^c]}^\T A_{[\Gamma^c]}} \le 1-\beta^2,
\end{equation*}
which implies
\begin{equation} \label{eq:lowbound}
\norm{\hat{\vb}_{[\Gamma]}-D_{S^*}A_{{[\Gamma]}} (D_{S^*}A_{{[\Gamma]}})^\T \hat{\vb}_{[\Gamma]}}
 \ge \beta^2 \norm{\hat{\vb}_{[\Gamma]}}.
\end{equation}
Putting (\ref{eq:medlow}) and (\ref{eq:lowbound}) into (\ref{eq:hs}), we have
\[
d(\KK_S,\vb)\ge 2\beta^2 \norm{\hat{\vb}_{[\Gamma]}} -\norm{(I-D_SA(D_SA)^\T)\eta}
 \ge 2\beta^2 \lambda-\norm{\eta}.
\]
Combining the above estimator with (\ref{eq:hsstar}) and noting $\norm{\eta} < \beta^2 \lambda$, we
obtain
\[
d(\KK_S,\vb) > \norm{\eta} \ge d(\KK_{S^*},\vb)
\]
for any $S \neq S^*$. Thus we complete the proof for the case where
$A^\T A =I$.

Finally, for general matrix $A\in \R^{m\times d}$,  we assume the singular
decomposition of  $A$ is $A=UDV^\T$ where $U:=[\vu_1,\ldots,\vu_m]^\T \in \R^{m\times
d}$ and $D \in \R^{d\times d}$ is an invertible  diagonal matrix. Let $\tA=U$, $\tx=D
V^\T \vx$ and $\tx_0=D V^\T \vx_0$. Then we have
\[
\norm{\abs{A\vx}-\vb}=\norm{\abs{\tA\tx}-\vb} \quad \mbox{and} \quad \vb=\abs{A\vx_0}+\eta=
\abs{\tA\tx_0}+\eta.
\]
Note that the program (\ref{eq:least}) has a unique solution  iff
$\mbox{argmin}_{\tx}\norm{\abs{\tA\tx}-\vb}$ has a unique solution. Hence,  if $\tA$
has $\beta$-strong complement property then we arrive at the conclusion. Indeed, for
any $\Gamma \subset \dkh{1,\ldots, m}$
\[
\tA_{[\Gamma]}^\T \tA_{[\Gamma]}=\sum_{i\in \Gamma} \vu_i\vu_i^\T \succeq \frac{1}{\sigma_{\max}^2(A)}  \cdot \sum_{i\in \Gamma} \sigma_i^2\ \vu_i\vu_i^\T= \frac{1}{\sigma_{\max}^2(A)}  \cdot A_{[\Gamma]}A_{[\Gamma]}^\T,
\]
where $\sigma_i$ are the singular values of $A$. Since $A\in \R^{m\times d}$ has
$\beta\sigma_{\max}(A)$-strong complement property, it immediately gives that
\[
\max\dkh{\lambda_{\min} (\tA_{[\Gamma]}^\T \tA_{[\Gamma]}), \lambda_{\min} (\tA_{[\Gamma^c]}^\T \tA_{[\Gamma^c]})} \ge \beta^2.
\]
Hence, $\tA$ has $\beta$-strong complement property. This completes the proof.
\end{proof}

Theorem \ref{th:clouniq} requires the matrix  $A$ has strong complement property.
The next lemma shows that the Gaussian random matrix satisfies such property with
high probability.
\begin{lemma}\cite[Theorem 20]{bandeira2014saving} \label{le:gauscp}
Assume that $A\in \R^{m\times d}$ is a Gaussian random matrix with  independent
standard normal entries. If $m>2d$ then for every $\epsilon>0$ the matrix $A$ has
$\sigma$-strong complement property with probability at least $1-\exp(-\epsilon m)$
where $R=m/d$ and
\[
\sigma=\frac{1}{\sqrt{2}e^{1+\epsilon/(R-2)}}\cdot \frac{m-2d+2}{2^{R/(R-2)}\sqrt{m}}.
\]
\end{lemma}

Combining Theorem \ref{th:clouniq} and Lemma \ref{le:gauscp}, we have the following
corollary.
\begin{cor}
Assume $m\ge Cd$.  Let $A\in \R^{m\times d}$ be a Gaussian random matrix with
independent standard normal entries. Suppose $\vb=\abs{A\vx_0}+\eta\in {\mathbb R}_+^m$ for some $\vx_0\in
\R^d$ and $\eta \in \R^m$.  Let $\lambda:=\min_i\dkh{\abs{A\vx_0}_i}$. If $\norm{\eta}\le c\lambda$ then with
probability at least $1-\exp(-c_0 m)$ the program (\ref{eq:least}) has a
unique solution, where $C, c$ and $c_0$ are universal constants.
\end{cor}
\begin{proof}
Picking $\epsilon=1$ in Lemma \ref{le:gauscp}, we obtain for $m>2d$, with probability at least $1-\exp(- m)$,
the matrix $A$ has $\sigma$-strong complement property with
\begin{equation} \label{eq:strocom}
\sigma=\frac{1}{\sqrt{2}e^{1+1/(R-2)}}\cdot \frac{m-2d+2}{2^{R/(R-2)}\sqrt{m}},
\end{equation}
where $R=m/d$. A simple observation is that  if $m\ge 4d$ then
\begin{equation*}
 \frac{1}{8\sqrt{2}e^{3/2}} \cdot \sqrt{m}\le \sigma \le  \frac{1}{2\sqrt{2}e } \cdot \sqrt{m}.
\end{equation*}

On the other hand,  since $A$ is a Gaussian random matrix, with probability at least $1-\exp(- cm)$, we have
\begin{equation} \label{eq:comsma}
\sqrt{m} \le \sigma_{\max}(A) \le 3 \sqrt{m}
\end{equation}
 provided $m\ge Cd$, where $C$ and $c$ are universal constants.

Combining \eqref{eq:strocom} with \eqref{eq:comsma}, we obtain that for $m\ge Cd$,
with probability at least $1-\exp(- c_0 m)$,  the matrix $A$ has
$\beta\sigma_{\max}(A)$-strong complement property for some constant $0<\beta<1$,
where $c_0$ is a universal constant. Then Theorem \ref{th:clouniq} implies the
conclusion holds.
\end{proof}

\section{Stability of solutions to (\ref{eq:least})}
In this section, we focus on Question III, i.e.,  the stability of solutions to
(\ref{eq:least}).
 For convenience, we set
\[
U_A:=\{\vb\in \R^m: \# P_{\KK_A}(\vb) =1\},
\]
where $P_{\KK_A}(\vb)$ is the set of the best approximation to  $\vb$ from $\KK_A$ and   $\# P_{\KK_A}(\vb)$ is the cardinality of $P_{\KK_A}(\vb)$. Recall  that  the program  (\ref{eq:least}) is
\begin{equation*}
\Phi_A(\vb)\,\,:=\,\,\argmin{\vx\in \HH^d}\|\abs{A\vx}-\vb\|^2.
\end{equation*}
The following lemma states $\Phi_A(\vb)$  is bilipschitz if $\vb \in \KK_A$.

\begin{lemma}\label{th:stable}\cite{phase}
Assume that $A\in \HH^{m\times d}$ has phase retrieval property. There exist
constants  $\alpha>0, \beta>0$ which only depend on $A$ so that
\begin{equation}\label{eq:stable}
\alpha\cdot \dist(\Phi_A(\vb),\Phi_A(\vb'))\,\,\leq \,\,
 \|\vb-\vb'\|\,\, \leq\,\, \beta \cdot \dist(\Phi_A(\vb),\Phi_A(\vb'))
\end{equation}
holds for all $\vb,\vb'\in \KK_A$.
\end{lemma}
Recall that $\vb \in \KK_A$ means $\vb=\abs{A\vx_0} $ for some $\vx_0$.
Hence, Lemma \ref{th:stable} shows that the phase retrieval problem in the absence of noise is stable.

\subsection{Instability of solutions to (\ref{eq:least})}
Lemma \ref{th:stable} only consider the case where $\vb \in \KK_A$. In this
subsection, we focus on  the general case where $\vb \notin \KK_A$ showing the
program (\ref{eq:least}) is unstable without the assumption of  $\vb \in \KK_A$. To
begin with,  we need the following lemma.
\begin{lemma}\label{th:nonsta}
Assume $A\in \HH^{m\times d}$ has phase retrieval property. For arbitrary constant
$\epsilon>0$ there exist $\vb_1,\vb_2\in U_A\subset \R^m$ so that
\[
{\|\vb_1-\vb_2\|}<\epsilon \cdot {\|P_{\KK_A}(\vb_1)-P_{\KK_A}(\vb_2)\|}.
\]
\end{lemma}
\begin{proof}
According to Theorem \ref{th:main}, there exists $\vb_0\in \R^m$ so that
$\#P_{\KK_A}(\vb_0)\geq 2$.  Assume $\vy_1,\vy_2\in P_{\KK_A}(\vb_0)$ with $\vy_1\neq
\vy_2$. We claim $P_{\KK_A}(\lambda \vy_1+(1-\lambda)\vb_0)=\{\vy_1\}$ and
$P_{\KK_A}(\lambda \vy_2+(1-\lambda)\vb_0)=\{\vy_2\}$ for any $\lambda\in (0,1)$.
Indeed, if $P_{\KK_A}(\lambda_0 \vy_1+(1-\lambda_0)\vb_0)$ contains some vector $\vy_0\in \KK_A$
with $\vy_0\neq \vy_1$ for some $\lambda_0\in (0,1)$ then
\begin{eqnarray*}
\|\vb_0-\vy_0\|&<&\|\vb_0-(\lambda_0 \vy_1+(1-\lambda_0)\vb_0)\|+\|(\lambda_0 \vy_1+(1-\lambda_0)\vb_0)-\vy_0\|\\
&\leq & \lambda_0 \|\vb_0-\vy_1\|+\|(\lambda_0 \vy_1+(1-\lambda_0)\vb_0)-\vy_1\|\\
&= & \lambda_0 \|\vb_0-\vy_1\|+ (1-\lambda_0) \|\vb_0-\vy_1\\\
&=& \|\vb_0-\vy_1\|,
 \end{eqnarray*}
 which contradicts to the fact that $\vy_1$ is a best approximation to $\vb_0$ from $\KK_A$.
 Hence, we immediately  obtain $P_{\KK_A}(\lambda \vy_1+(1-\lambda)\vb_0)=\{\vy_1\}$. Similarly, we
 can show $P_{\KK_A}(\lambda \vy_2+(1-\lambda)\vb_0)=\{\vy_2\}$.
  Taking $\vb_1=\frac{\epsilon}{2} \vy_1+(1-\frac{\epsilon}{2})\vb_0\in U_A$ and
$\vb_2=\frac{\epsilon}{2} \vy_2+(1-\frac{\epsilon}{2})\vb_0\in U_A$, we arrive at the conclusion that
\[
\|\vb_1-\vb_2\|=\frac{\epsilon}{2}\|\vy_1-\vy_2\|<\epsilon \|\vy_1-\vy_2\|=
\epsilon \cdot {\|P_{\KK_A}(\vb_1)-P_{\KK_A}(\vb_2)\|}.
\]
\end{proof}

Combining Lemma  \ref{th:stable} and Lemma \ref{th:nonsta}, we have the  following instability result.
\begin{theorem}
Assume that $A\in \HH^{m\times d}$ has phase retrieval property. Then for any
$\epsilon>0$ there exist $\vb_1,\vb_2\in U_A$ so that
\[
\|\vb_1-\vb_2\|\,\,\leq \,\,\epsilon\cdot \dist(\Phi_A(\vb_1),\Phi_A(\vb_2)).
\]
\end{theorem}
\begin{proof}
According to Lemma \ref{th:nonsta}, there exist $\vb_1,\vb_2\in U_A$ so that
\begin{equation}\label{eq:uns}
{\|\vb_1-\vb_2\|}<\frac{\epsilon}{\beta} \cdot {\|P_{\KK_A}(\vb_1)-P_{\KK_A}(\vb_2)\|},
\end{equation}
where $\beta$ is the constant given in Lemma \ref{th:stable}. Note that
$P_{\KK_A}(\vb_j)\in \KK_A$ and $\Phi_A(P_{\KK_A}(\vb_j))=\Phi_A(\vb_j), j=1,2 $.
Combining the right hand sides of (\ref{eq:stable}) and (\ref{eq:uns}), we arrive at the
conclusion.
\end{proof}

The above theorem shows that the program (\ref{eq:least}) is not uniformly stable.
However, we next show that if we restrict the vector $\vb$ to any convex set
$\Omega\subset U_A$ then it is stable. To show that, we introduce a  lemma first.

\begin{lemma}\label{le:cont}
Assume that $\Omega\subset U_A$ is a convex domain. For any matrix $A\in \HH^{m\times
d}$, the $P_{\KK_A}(\cdot)$ is continuous on $\Omega$.
\end{lemma}
\begin{proof}
For the aim of  contradiction, we assume that there is a point $\vx\in \Omega$ and a
sequence $\{\vx_j\}_{j=1}^\infty \subset \Omega$ with $\lim_{j\to \infty}\vx_j=\vx$
such that for every $j$ it holds
\begin{equation}\label{eq:con1}
\|P_{\KK_A}(\vx_j)-P_{\KK_A}(\vx)\| \geq \epsilon >0.
\end{equation}
From the definition, we have
\begin{equation}\label{eq:d1}
d(\KK_A,\vx)\leq \|\vx-P_{\KK_A}(\vx_j)\|\leq \|\vx-\vx_j\|+\|\vx_j-P_{\KK_A}(\vx_j)\|=
\|\vx-\vx_j\|+d(\KK_A,\vx_j).
\end{equation}
According to Lemma \ref{le:dist}, we know $d(\KK_A,\vx_j)$ converges to
$d(\KK_A,\vx)$.  By squeeze theorem,  (\ref{eq:d1}) implies
$\|\vx-P_{\KK_A}(\vx_j)\|$ converges to $d(\KK_A,\vx) $. Note that
$\{P_{\KK_A}(\vx_{j})\}_{j=1}^\infty$ is a bounded sequence.  There exists a
subsequence which is convergent.  We assume $P_{\KK_A}(\vx_{j_t})$ converges to
$\vy\in \R^m$. Then
\[
\|\vx-\vy\|=\lim_{t\rightarrow \infty}\|\vx-P_{\KK_A}(\vx_{j_t})\|=d(\KK_A,\vx).
\]
It implies $P_{\KK_A}(\vx)=\{\vy\}$. Hence,
\[
\lim_{t\rightarrow \infty}\|P_{\KK_A}(\vx_{j_t})-P_{\KK_A}(\vx)\|=
\lim_{t\rightarrow \infty}\|P_{\KK_A}(\vx_{j_t})-\vy\|=0,
\]
which contradicts to (\ref{eq:con1}).
\end{proof}

Now, we could extend the stability result in Lemma \ref{th:stable}  from $\vb\in
\KK_A$ to any convex set $\Omega\subset U_A$ in the real case. We also conjecture a
similar result holds for the complex case.
\begin{theorem}
Assume that $A\in \R^{m\times d}$ has phase retrieval property in $\R^d$.
 Let $\Omega\subset U_A$ be a convex domain. Then there exists a constant $\alpha>0$ which only
depends on $A$ so that
\begin{equation}\label{eq:thst}
\alpha \cdot \dist(\Phi_A(\vb),\Phi_A(\vb'))\,\,\leq \,\,
 \|\vb-\vb'\| \quad \text{ for all }\quad \vb,\vb'\in \Omega.
\end{equation}
\end{theorem}
\begin{proof}
 Assume that $\vb,\vb'\in \Omega$. If $\vb=\vb'$, then the conclusion
(\ref{eq:thst}) holds. Thus, we just need to consider the case where $\vb\neq \vb'$.
Note that $P_{\KK_A}(\vb), P_{\KK_A}(\vb')\in \KK_A$. According to Lemma
\ref{th:stable}, there exists a constant $\alpha>0$ which only depends on $A$ so that
\[
 \alpha\cdot \dist(\Phi_A(\vb),\Phi_A(\vb'))\,\,=\,\,\alpha\cdot \dist(\Phi_A(P_{\KK_A}(\vb)),\Phi_A(P_{\KK_A}
 (\vb'))) \,\,\leq \,\,
 \|P_{\KK_A}(\vb)-P_{\KK_A}(\vb')\| .
\]
So, to prove the conclusion, it is enough to show that
\[
 \|P_{\KK_A}(\vb)-P_{\KK_A}(\vb')\| \leq \|\vb-\vb'\|.
\]
 Set
\[
[\vb,\vb']\,\,:=\,\, \{(1-\lambda) \vb+\lambda\vb':\lambda\in [0,1]\}.
\]
Since $\Omega$ is convex, we have $[\vb,\vb']\subset \Omega\subset U_A$. According to
Lemma \ref{le:cont}, $P_{\KK_A}$ is continuous on $[\vb,\vb']$. We use
$D_\bepsilon\in \R^{m\times m}$ to denote a diagonal matrix whose diagonal is
$\bepsilon\in \{1,-1\}^m$. Set
\[
H_\bepsilon:=\{D_\epsilon A\vx\geq 0: \vx\in \R^d\}.
\]
A simple observation is that $H_\bepsilon\subset \KK_A$ is convex for any fixed
$\bepsilon \in \{1,-1\}^m$. We assume that $(\vb_0,\vb_1,\ldots,\vb_k)$ is a
partition of $[\vb,\vb']$ such that $\vb=\vb_0<\vb_1<\cdots<\vb_k=\vb'$ and $P_{\KK_A}([\vb_{t-1},\vb_{t}])\subset H_{\epsilon_t}$ where $\epsilon_t\in
\{1,-1\}^m$, $t\in \{1,\ldots,k\}$. Since $P_{\KK_A}$ is continuous on
$[\vb,\vb']$, it means $P_{\KK_A}(\vb_t)\in H_{\epsilon_t}\cap H_{\epsilon_{t+1}}$ for
$t=1,2,\ldots,k-1$. According to $P_{\KK_A}([\vb_{t-1},\vb_{t}])\subset
H_{\epsilon_t}$, we have $P_{\KK_A}([\vb_{t-1},\vb_{t}])=
P_{H_{\epsilon_t}}([\vb_{t-1},\vb_{t}])$. Note that $H_{\epsilon_t}$ is convex. It then follows from Lemma \ref{le:leq1}  that
\[
\|P_{\KK_A}(\vb_t)-P_{\KK_A}(\vb_{t-1})\|\,\,
=\,\, \|P_{H_{\epsilon_t}}(\vb_t)-P_{H_{\epsilon_t}}(\vb_{t-1})\| \,\,\leq \,\, \|\vb_t-\vb_{t-1}\|.
\]
Thus we have
\begin{equation}\label{eq:bu1}
\begin{aligned}
\|P_{\KK_A}(\vb)-P_{\KK_A}(\vb')\|\leq \sum_{t=1}^k\|P_{\KK_A}(\vb_t)-P_{\KK_A}(\vb_{t-1})\|
\leq \sum_{t=1}^k  \|\vb_t-\vb_{t-1}\|=\|\vb-\vb'\|,
\end{aligned}
\end{equation}
where the last equation follows from $\vb_t,t=0,1,\ldots,k,$ are collinear  points.
We arrive at the conclusion.
\end{proof}

\noindent {\bf Acknowledgments.}~ Zhiqiang Xu  is most grateful to Yang Wang for
discussions and comments which are helpful for the proof of Theorem 3.1. M. Huang
acknowledges support from Yang Wang and the Department of Mathematics, The Hong Kong
University of Science and Technology.

%
%
%
%


\end{document}